\documentclass[english]{amsart}

\usepackage{esint}
\usepackage[svgnames]{xcolor} 
\usepackage[colorlinks,citecolor=red,pagebackref,hypertexnames=false,breaklinks]{hyperref}
\usepackage{pgf,tikz}
\usepackage{pdfsync}

\usepackage{dsfont}
\usepackage{url}
\usepackage[utf8]{inputenc}
\usepackage[T1]{fontenc}
\usepackage{lmodern}
\usepackage{babel}
\usepackage{mathtools}  
\usepackage{amssymb}
\usepackage{lipsum}
\usepackage{mathrsfs}
\usepackage{color}
\usepackage[skip=4pt plus1pt, indent=12pt]{parskip}

\newtheorem{theorem}{Theorem}[section]
\newtheorem{proposition}{Proposition}[section]
\newtheorem{lemma}{Lemma}[section]

\numberwithin{equation}{section}

\title[An inverse spectral problem]{An inverse spectral problem for a fractional Schr\"odinger operator}

\author[Mourad Choulli]{Mourad Choulli}
\address{Universit\'e de Lorraine}
\email{mourad.choulli@univ-lorraine.fr}

\thanks{The author is supported by the grant ANR-17-CE40-0029 of the French National Research Agency ANR (project MultiOnde). }

\date{}

\begin{document}

\begin{abstract}
We establish that the potential appearing in a fractional Schrödinger operator is uniquely determined by an  internal spectral data.
 \end{abstract}



\maketitle


\section{Introduction}

Let $\mathcal{M}=(\mathcal{M},\mathfrak{g})$ be a smooth connected compact Riemannian manifold with boundary $\partial \mathcal{M}$ of dimension $n\ge 2$. Recall that the Laplace-Beltrami operator associated to the metric tensor $\mathfrak{g}=(\mathfrak{g}_{k\ell})$ is given in local coordinates by
\[
\Delta_\mathfrak{g}=\frac{1}{\sqrt{|\mathfrak{g}|}}\sum_{k,\ell=1}^n\frac{\partial}{\partial x_k}\left(\sqrt{|\mathfrak{g}|}\mathfrak{g}^{k\ell}\frac{\partial}{\partial x_\ell}\, \cdot \right),
\]
where $(\mathfrak{g}^{k\ell})$ is the matrix inverse of $(\mathfrak{g}_{k\ell})$ and $|\mathfrak{g}|$ is the determinant of $\mathfrak{g}$.

The following notation will be useful in the sequel
\[
H_D^{\theta}(\mathcal{M})=\{u\in H^\theta(\mathcal{M});\;  u_{|\partial \mathcal{M}}=0\},\quad \theta >1/2.
\]

We denote by $A$ the realization on $L^2(\mathcal{M})$ of the negative Laplace-Beltrami operator with Dirichlet boundary conditions, that is, $D(A)=H_D^2(\mathcal{M})$ and $A=-\Delta _{\mathfrak{g}}$.

As $A$ is positive definite self-adjoint operator with compact resolvent its spectrum is reduced to a sequence of eigenvalues counted according to their multiplicity:
\[
0 < \lambda_1<\lambda_2\le  \ldots \lambda_k\le \ldots \quad \mbox{and}\quad \lambda_k\rightarrow \infty \; \mbox{as}\; k\rightarrow \infty.
\]

Moreover there exists $(\phi_k)_{k\ge 1}$ an orthonormal basis of $L^2(\mathcal{M})$ consisting of eigenfunctions or more precisely, each $\phi_k$ is an eigenfunction for the eigenvalue $\lambda_k$.

Let $0< s\le  1$. We recall that the fractional power $A^s$ is defined as follows
\[
A^su=\sum_{k\ge 1}\lambda_k^s(u|\phi_k)\phi_k,\quad u\in D(A^s).
\]
Here and henceforth the symbol $(\cdot|\cdot)$ denotes the natural scalar product of $L^2(\mathcal{M})$.

Following \cite[Theorem 1]{Fu} (see also \cite[Theorem 2.1]{Na}) we have
\[
D(A^s):=\mathcal{H}^s=\left\{
\begin{array}{ll}
\displaystyle H^{2s}(\mathcal{M}),\quad 0< s<1/4, \\ \\ \displaystyle \left\{u\in H^{1/2}(\mathcal{M});\int_{\mathcal{M}}\zeta^{-1}(x)|u(x)|^2dV<\infty\right\}, \quad s=1/4,\\ \\ \displaystyle
H_D^{2s}(\mathcal{M}),\quad 1/4<s\le 1,
\end{array}
\right.
\]
where $dV=\sqrt{|\mathfrak{g}|}dx^1\ldots dx^n$ denotes the Riemannian measure associated to $\mathfrak{g}$ and $\zeta$ is the distance to the boundary $\partial \mathcal{M}$.

It is worth noticing that, contrary to $A$, the operator $A^s$, $0<s<1$, is nonlocal. 

Set
\[
L_+^\infty(\mathcal{M})=\{ \mathfrak{q}\in L^\infty(\mathcal{M},\mathbb{R});\; \mathfrak{q}\ge 0\}
\]
For $\mathfrak{q}\in L_+^\infty(\mathcal{M})$ we denote by $A_{\mathfrak{q}}$ the realization on $L^2(\mathcal{M})$ of the operator $A^s+\mathfrak{q}$ with domain $D(A_{\mathfrak{q}})=\mathcal{H}^s$.

As $\mathcal{H}^s$ is compactly embedded in $L^2(\mathcal{M})$ and $A_{\mathfrak{q}}$ is self-adjoint we derive that the spectrum of $A_{\mathfrak{q}}$, denoted by $\sigma(A_{\mathfrak{q}})$, consists in a nondecreasing sequence $\left(\mu_k^{\mathfrak{q}}\right)$ converging to $\infty$. Furthermore to each $\mu_k^{\mathfrak{q}}$ we can associate an eigenfunction $\varphi_k^{\mathfrak{q}}$ in such a way that $\left(\varphi_k^{\mathfrak{q}}\right)$ form an orthonormal basis of $L^2(\mathcal{M})$.

Let $\Omega_0,\Omega_1$ and $\Omega'$ be three nonempty open subsets of $\mathrm{Int}(\mathcal{M})$ chosen in such a way that $\Omega_j\setminus \overline{\Omega'}\not=\emptyset$, $j=0,1$. Set $\Omega=\Omega_0\cup \Omega_1$ and consider the following subset of $L_+^\infty(\mathcal{M})\times L_+^\infty(\mathcal{M})$ 
\[
\mathcal{Q}=\left\{ (\mathfrak{q}_1,\mathfrak{q}_2)\in L_+^\infty(\mathcal{M})\times L_+^\infty(\mathcal{M});\; \mathrm{supp}(\mathfrak{q}_1-\mathfrak{q}_2)\subset \Omega'\right\}.
\]

\begin{theorem}\label{main-theorem}
Let $(\mathfrak{q}_1,\mathfrak{q}_2) \in \mathcal{Q}$. If
\[
\mu_k ^{\mathfrak{q}_1}=\mu_k^{\mathfrak{q}_2},\quad \varphi_k^{\mathfrak{q}_1}{_{|\Omega}}=\varphi_k^{\mathfrak{q}_2}{_{|\Omega}},\quad k\ge 1,
\]
then $\mathfrak{q}_1=\mathfrak{q}_2$.
\end{theorem}

As $\mu_k ^{\mathfrak{q}_j+\mu}=\mu_k ^{\mathfrak{q}_j}+\mu$, for any $k\ge1$, $\mu \in \mathbb{R}$ and $j=1,2$, it is clear that we can remove the condition $\mathfrak{q}_j\ge 0$ in Theorem \ref{main-theorem}.

The case $s=1$ was considered in \cite{HLOS} when $(\mathcal{M},\mathfrak{g})$ is a smooth connected compact Riemannian manifold without boundary. The authors established in \cite[Corollary 2]{HLOS} that the internal spectral data $\left(\lambda_k^{\mathfrak{g}}, \phi_k^{\mathfrak{g}}{_{|\Omega}}\right)$, where $\left(\lambda_k^{\mathfrak{g}}, \phi_k^{\mathfrak{g}}\right)$ is the sequence of eigenvalues and the corresponding eigenfunctions of $-\Delta_{\mathfrak{g}}$ and $\Omega$ is an open subset of $\mathrm{Int}(\mathcal{M})$, determines uniquely $\mathfrak{g}$ up to isometry.

There is a wide literature devoted to inverse spectral problems for Schr\"odinger operators and magnetic Schr\"odinger operators involving boundary spectral data. We refer to the recent work \cite{BCDKS} and references therein.

\section{Proof of Theorem \ref{main-theorem}}

We give the proof in several steps. The key step is based on the density result in Lemma \ref{lemma3.1} which is inspired by the proof of \cite[Theorem 1.1]{GSU}.

The resolvent of $A_{\mathfrak{q}}$, $\mathfrak{q}\in L_+^\infty(\mathcal{M})$, is denoted by $R_{\mathfrak{q}}$:
\[
R_{\mathfrak{q}}:\mu \in \rho(A_{\mathfrak{q}})\mapsto R_{\mathfrak{q}}(\mu)=\left(A_{\mathfrak{q}}-\mu\right)^{-1}\in \mathscr{B}\left(L^2(\mathcal{M}\right), \mathcal{H}^s),
\]
where $\rho(A_{\mathfrak{q}})=\mathbb{C}\setminus \sigma(A_{\mathfrak{q}})$ is the resolvent set of $A_{\mathfrak{q}}$.

We can then assert that, for each $f\in L^2(\mathcal{M})$ and $\mu \in \rho(A_{\mathfrak{q}})$, $u=R_{\mathfrak{q}}(\mu)f\in \mathcal{H}^s$ is the unique variational solution of the equation
\begin{equation}\label{bvp1}
\left(A^s+\mathfrak{q}-\mu\right)u=f.
\end{equation}
That is, we have
\begin{equation}\label{i1}
\left(A^{s/2}u|A^{s/2}v\right)+((\mathfrak{q}-\mu)u|v)=(f|v),\quad v\in  \mathcal{H}^{s/2}.
\end{equation}

As $\varphi_k^{\mathfrak{q}}=R_{\mathfrak{q}}(0)\left(-\mathfrak{q}\varphi_k^{\mathfrak{q}}+\mu_k^{\mathfrak{q}}\varphi_k^{\mathfrak{q}}\right)\in \mathcal{H}^s$, $k\ge 1$, we have
\begin{align*}
\left(u|\mu_k^{\mathfrak{q}}\varphi_k^{\mathfrak{q}}\right)&=\left(u|A^s\varphi_k^{\mathfrak{q}}\right)+\left(u|\mathfrak{q}\varphi_k^{\mathfrak{q}}\right)
\\
&=\left(A^{s/2}u|A^{s/2}\varphi_k^{\mathfrak{q}}\right)+\left(u|\mathfrak{q}\varphi_k^{\mathfrak{q}}\right).
\end{align*}
This and \eqref{i1} imply
\[
\left(u|\mu_k^{\mathfrak{q}}\varphi_k^{\mathfrak{q}}\right)=\left(f|\varphi_k^{\mathfrak{q}})-((\mathfrak{q}-\mu)u|\varphi_k^{\mathfrak{q}}\right)+\left(u|\mathfrak{q}\varphi_k^{\mathfrak{q}}\right)
\]
from which we derive in a straightforward manner
\begin{equation}\label{i2}
\left(u|\varphi_k^{\mathfrak{q}}\right)=\left(\mu_k^{\mathfrak{q}}-\mu\right)^{-1}\left(f|\varphi_k^{\mathfrak{q}}\right).
\end{equation}
Hence
\begin{equation}\label{i3}
R_{\mathfrak{q}}(\mu)f=\sum_{k\ge 1}\left(\mu_k^{\mathfrak{q}}-\mu\right)^{-1}\left(f|\varphi_k^{\mathfrak{q}}\right)\varphi_k^{\mathfrak{q}}.
\end{equation}

Note that since $\min_{k\ge 1}\left|\mu_k^{\mathfrak{q}}-\mu\right|=\mathrm{dist}\left(\mu,\sigma\left(A_{\mathfrak{q}}\right)\right)$ it follows from \eqref{i3} the following well-known resolvent estimate
\begin{equation}\label{r1}
\left\|R_{\mathfrak{q}}(\mu)\right\|_{\mathscr{B}(L^2(\mathcal{M}))}\le \frac{1}{\mathrm{dist}\left(\mu,\sigma\left(A_{\mathfrak{q}}\right)\right)}.
\end{equation}

We now make an observation that will be used in the sequel.

According to Weyl's asymptotic formula we have $\lambda_k=O\left(k^{2/n}\right)$ (e.g. \cite{Be}). On the other hand we infer from the min-max principle that $\lambda_k^s\le \mu_k^{\mathfrak{q}}\le \lambda_k^s+m$, for every $q\in L^\infty(\mathcal{M})$ satisfying $0\le q\le m$. We derive that $\mu_k^{\mathfrak{q}}=O\left(k^{2s/n}\right)$, uniformly with respect to $q\in L^\infty(\mathcal{M})$ satisfying $0\le q\le m$.

For notational convenience $R_{\mathfrak{q}}(0)$ is simply denoted in the rest of this section by $R_{\mathfrak{q}}$. Also, we consider $L^2(\Omega_0)$ as a subset of $L^2(\mathcal{M})$:
\[
L^2(\Omega_0)=\left\{f\in L^2(\mathcal{M});\; \mathrm{supp}(f)\subset \Omega_0\right\}
\]
and for $\mathfrak{q}\in L_+^\infty(\mathcal{M})$ define the source-to-solution operator $\Sigma_{\mathfrak{q}}$ as follows
\[
\Sigma_{\mathfrak{q}}: L^2(\Omega_0)\rightarrow L^2(\Omega_1):f\mapsto R_{\mathfrak{q}}f{_{|\Omega_1}}.
\]

Pick $\mathfrak{q}_j\in L^\infty(\mathcal{M})$ satisfying $0\le \mathfrak{q}_j\le m$, $j=1,2$. Since
\[
\Sigma_{\mathfrak{q}_j}f=\sum_{k\ge 1}\left(\mu_k^{\mathfrak{q}_j}\right)^{-1}(f|\varphi_k^{\mathfrak{q}_j})\varphi_k^{\mathfrak{q}_j}{_{|\Omega_1}},\quad j=1,2,
\]
we have
\[
\Sigma_{\mathfrak{q}_1}f-\Sigma_{\mathfrak{q}_2}f=I_1+I_2+I_3,\quad f\in L^2(\Omega_0),
\]
where
\begin{align*}
&I_1=\sum_{k\ge 1}\left[\left(\mu_k^{\mathfrak{q}_1}\right)^{-1}-\left(\mu_k^{\mathfrak{q}_2}\right)^{-1}\right]\left(f|\varphi_k^{\mathfrak{q}_1}\right)\varphi_k^{\mathfrak{q}_1}{_{|\Omega_1}},
\\
&I_2=\sum_{k\ge 1}\left(\mu_k^{\mathfrak{q}_2}\right)^{-1}\left(f|\varphi_k^{\mathfrak{q}_1}-\varphi_k^{\mathfrak{q}_2}\right)\varphi_k^{\mathfrak{q}_1}{_{|\Omega_1}},
\\
&I_3=\sum_{k\ge 1}\left(\mu_k^{\mathfrak{q}_2}\right)^{-1}\left(f|\varphi_k^{\mathfrak{q}_2}\right)\left(\varphi_k^{\mathfrak{q}_1}{_{|\Omega_1}}-\varphi_k^{\mathfrak{q}_2}{_{|\Omega_1}}\right).
\end{align*}

If $\|f\|_{L^2(\mathcal{M})}=1$ then
\begin{equation}\label{2.1}
\|I_1\|_{L^2(\Omega_1)}\le C\sum_{k\ge 1}k^{-4s/n}\left|\mu_k^{\mathfrak{q}_1}-\mu_k^{\mathfrak{q}_2}\right|
\end{equation}
and
\begin{equation}\label{2.2}
\|I_{2+j}\|_{L^2(\Omega_1)}\le C\sum_{k\ge 1}k^{-2s/n}\left\|\varphi_k^{\mathfrak{q}_1}-\varphi_k^{\mathfrak{q}_2}\right\|_{L^2(\Omega_j)},\quad j=0,1.
\end{equation}
Here and henceforth $C>0$ is a generic constant only depending on $\mathcal{M}$, $\mathfrak{g}$ and $m$.

Define
\[
\mathbf{d}(\mathfrak{q}_1,\mathfrak{q}_2)=\sum_{k\ge 1}\left[k^{-4s/n}\left|\mu_k^{\mathfrak{q}_1}-\mu_k^{\mathfrak{q}_2}\right|+k^{-2s/n}\left\|\varphi_k^{\mathfrak{q}_1}-\varphi_k^{\mathfrak{q}_2}\right\|_{L^2(\Omega)}\right]
\]
and suppose that $\mathbf{d}(\mathfrak{q}_1,\mathfrak{q}_2)<\infty$. Then inequalities \eqref{2.1} and \eqref{2.2} yield
\begin{equation}\label{2.3}
\left\|\Sigma_{\mathfrak{q}_1}-\Sigma_{\mathfrak{q}_2}\right\|_{\mathscr{B}\left(L^2(\Omega_0),L^2(\Omega_1)\right)}\le C\mathbf{d}(\mathfrak{q}_1,\mathfrak{q}_2).
\end{equation}

An immediate consequence of \eqref{2.3} is

\begin{proposition}\label{proposition2}
Let $\mathfrak{q}_1,\mathfrak{q}_2\in L_+^\infty (\mathcal{M})$. If
\[
\mu_k^{\mathfrak{q}_2}=\mu_k^{\mathfrak{q}_1},\quad \varphi_k^{\mathfrak{q}_2}{_{|\Omega}}=\varphi_k^{\mathfrak{q}_1}{_{|\Omega}} \quad k\ge 1,
\]
then $\Sigma_{\mathfrak{q}_1}=\Sigma_{\mathfrak{q}_2}$.
\end{proposition}

Next, we state a result borrowed to \cite[Theorem 1.5]{Yu} which  still holds for $H^{2s}$-solutions. We point out that the reflexion extension argument used in the proof of  \cite[Theorem 1.5]{Yu} is contained in \cite[Theorem 3.1]{SZ}.

\begin{theorem}\label{theorem3.1}
Pick $V$ a nonempty open subset of $\mathrm{Int}(\mathcal{M})$. If $u\in \mathcal{H}^s$ satisfies $u=A^su=0$ in $V$ then $u=0$.
\end{theorem}

Theorem  \ref{theorem3.1} will serve to establish the following density result.

\begin{lemma}\label{lemma3.1}
Let $\omega_0$, $\omega_1$ be two nonempty open subsets of $\mathrm{Int}(\mathcal{M})$ so that $V=\omega_0\setminus \overline{\omega}_1\not=\emptyset$. For every $q\in L_+^\infty (\mathcal{M})$ the set
\[
\mathcal{S}_{\mathfrak{q}}=\left\{u=R_{\mathfrak{q}}f_{|\omega_1};\; f\in C_0^\infty (\omega_0)\right\}
\]
is dense in $L^2(\omega_1)$.
\end{lemma}

\begin{proof}
Let $g\in L^2(\omega_1)$, extended by $0$ outside $\omega_1$, satisfy $(g|u)=0$ for any $u\in \mathcal{S}_q$. Equivalently we have 
\begin{equation}\label{3.1}
\left(g|R_{\mathfrak{q}}f\right)=0,\quad  f\in C_0^\infty (\omega_0).
\end{equation}
If $v=R_qg$ then
\[
\left(g|R_{\mathfrak{q}}f\right)=\sum_{k\ge 1}\left(\mu_k^{\mathfrak{q}}\right)^{-1}\left(\varphi_k^{\mathfrak{q}}|f\right)\left(g|\varphi_k^{\mathfrak{q}}\right)=\left(R_{\mathfrak{q}}g|f\right)=(v|f),\quad  f\in C_0^\infty (\omega_0).
\]
Hence \eqref{3.1} yields
\[
(v|f)=0,\quad  f\in C_0^\infty (\omega_0).
\]
Therefore $v=0$ in $\omega_0$. Using $A^sv=g-qv$ and $\mathrm{supp}(g)\subset \omega_1$ we obtain $v=A^sv=0$ in $V$. Theorem \ref{theorem3.1} implies $v=0$ and hence $g=0$.
\end{proof}

\begin{proof}[Completion of the proof of Theorem \ref{main-theorem}]
Let $(\mathfrak{q}_1,\mathfrak{q}_2)\in \mathcal{Q}$ satisfying
\[
\mu_k ^{\mathfrak{q}_1}=\mu_k^{\mathfrak{q}_2},\quad \varphi_k^{\mathfrak{q}_1}{_{|\Omega}}=\varphi_k^{\mathfrak{q}_2}{_{|\Omega}},\quad k\ge 1.
\]
For $u_j=R_{\mathfrak{q}_j}f$, $f\in L^2(\Omega_0)$, we have
\[
A^s u_1+\mathfrak{q}_1u_1=A^s u_2+\mathfrak{q}_2u_2.
\]
Hence 
\[
A^s(u_1-u_2)+\mathfrak{q}_1(u_1-u_2)=(\mathfrak{q}_2-\mathfrak{q}_1)u_2.
\]
If $v_1=R_{\mathfrak{q}_1}g$, $g\in C_0^\infty (\Omega_1)$, then 
\[
\left(\left(u_1-u_2\right)|A^sv_1\right)+\left(\left(u_1-u_2\right)|\mathfrak{q}_1v_1\right)=\left(\left(\mathfrak{q}_2-\mathfrak{q}_1\right)u_2|v_1\right).
\]
Using that $A^sv_1+\mathfrak{q}_1v_1=g$ we find
\[
\left(\left(u_1-u_2\right)|g\right)= \left(\left(\mathfrak{q}_2-\mathfrak{q}_1\right)u_2|v_1\right).
\]
It follows from Proposition \ref{proposition2} that $u_1-u_2=0$ in $\Omega_1$. This and the fact that $\mathrm{supp}(g)\subset \Omega_1$ imply
\[
\left(\left(\mathfrak{q}_2-\mathfrak{q}_1\right)u_2|v_1\right)=0.
\]
In other words we proved
\begin{equation}\label{3.2}
\left(\left(\mathfrak{q}_2-\mathfrak{q}_1\right)R_{\mathfrak{q}_2}f|R_{\mathfrak{q}_1}g\right)=0,\quad f\in C_0^\infty (\Omega_0),\; g\in C_0^\infty (\Omega_1).
\end{equation}
As $\mathrm{supp}\left(\mathfrak{q}_2-\mathfrak{q}_1\right)\subset \Omega'$ and $\left\{R_{\mathfrak{q}_1}g{_{|\Omega'}};\; g\in C_0^\infty (\Omega_1)\right\}$ is dense $L^2(\Omega')$ by Lemma \ref{lemma3.1}  
we find  by applying \eqref{3.2} to a sequence $(g_k)$ in $C_0^\infty (\Omega_1)$ so that $R_{\mathfrak{q}_1}g_k{_{|\Omega'}}$ converges to $1$ in $L^2(\Omega')$
\[
\left(\left(\mathfrak{q}_2-\mathfrak{q}_1\right)R_{\mathfrak{q}_2}f|1\right)=\left(R_{\mathfrak{q}_2}f|\mathfrak{q}_2-\mathfrak{q}_1\right)=0,\quad f\in C_0^\infty (\Omega_0).
\]

Applying once again Lemma \ref{lemma3.1} (with $\omega_0=\Omega_0$ and $\omega_1=\Omega'$) we finally obtain $\mathfrak{q}_2=\mathfrak{q}_1$. 
\end{proof}

The preceding proof contains the following result
\begin{theorem}\label{theoremF}
Let $(\mathfrak{q}_1,\mathfrak{q}_2) \in \mathcal{Q}$. If  $\Sigma_{\mathfrak{q}_1}=\Sigma_{\mathfrak{q}_2}$, then $\mathfrak{q}_1=\mathfrak{q}_2$.
\end{theorem}

Let $(-\Delta)^s=\mathrm{Op}(|\xi|^{2s})$, $0<s<1$, be the fractional power of $-\Delta$ acting as an operator on $\mathbb{R}^n$. Pick $D$ a bounded domain of $\mathbb{R}^n$ and consider the following problem
\begin{equation}\label{p}
\left((-\Delta )^s+q\right)u=0\; \mathrm{in}\; D,\quad u_{|D_e}=f,
\end{equation}
where $D_e=\mathbb{R}^n\setminus\overline{D}$.

When $D$ and $\mathfrak{q}$ are sufficiently smooth the authors define in \cite{GSU}  an analogue of the so-called Dirichlet-to-Neumann map  as follows
\[
\Lambda_{\mathfrak{q}}:f\in H^{s+\beta}(D_e)\mapsto (-\Delta u(f))^s{_{|D_e}}\in H^{-s+\beta}(D_e),
\]
where $u(f)$ is the variational solution of \eqref{p} and $\max(0,s-1/2)<\beta<1/2$.

Let $W_1,W_2$ be two open subsets of $D_e$. It is proved in \cite[Theorem 1.1]{GSU} that $f\in C_0^\infty (W_1)\mapsto \Lambda_{\mathfrak{q}}f{_{|W_2}}$ determines uniquely $\mathfrak{q}$.

A generalization of this result was obtained in \cite{GLX} when $-\Delta$ is substituted by an elliptic operator with variable coefficients.

We refer to \cite[Theorem 1.2]{RS} to a stability inequality corresponding to the uniqueness result in \cite{GSU}.

\section{Comments}

Fix $\Omega$ a nonempty open subset of $\mathrm{Int}(\mathcal{M})$ and let $\Omega_0$ and $\Omega_1$ be two nonempty open subsets so that $\Omega_0\cup \Omega_1=\Omega$. Let $\mathfrak{q}_1,\mathfrak{q}_2 \in L_+^\infty (\mathcal{M})$. If the assumption
\begin{equation}\label{r0}
\mu_k ^{\mathfrak{q}_1}=\mu_k^{\mathfrak{q}_2},\quad \varphi_k^{\mathfrak{q}_1}{_{|\Omega}}=\varphi_k^{\mathfrak{q}_2}{_{|\Omega}},\quad k\ge 1,
\end{equation}
holds then we derive from formula \eqref{i3}
\begin{equation}\label{r1}
R_{\mathfrak{q}_1}(-\mu)f{_{|\Omega_1}}= R_{\mathfrak{q}_2}(-\mu)f{_{|\Omega_1}},\quad \Re \mu>0,\; f\in C_0^\infty(\Omega_0).
\end{equation}

This means that the internal spectral data $\left(\lambda_k^{\mathfrak{q}}, \varphi_k^{\mathfrak{q}}{_{|\Omega}}\right)$ does not only determine the source-to-solution operator $\Sigma_{\mathfrak{q}}$ but the whole family of source-to-solution operators defined as follows
\[
\Sigma_{\mathfrak{q}}(\mu):f\in C_0^\infty(\Omega_0)\mapsto R_{\mathfrak{q}}(-\mu)f{_{|\Omega_1}},\quad \Re \mu >0,
\]

It turns out that the knowledge of the family of source-to-solution operators $\left(\Sigma_{\mathfrak{q}}(\mu)\right)_{\Re \mu >0}$ determines the internal data $\left(\lambda_k^{\mathfrak{q}}, \varphi_k^{\mathfrak{q}}{_{|\Omega}}\right)$. Precisely, we have the following result.

\begin{proposition}\label{propositionco}
Suppose $\Omega_0\cap \Omega_1\not=\emptyset$. Let $\mathfrak{q}_1,\mathfrak{q}_2 \in L_+^\infty (\mathcal{M})$ satisfying
 \[
 \Sigma_{\mathfrak{q}_1}(\mu)=\Sigma_{\mathfrak{q}_2}(\mu),\quad \Re \mu >0.
 \]
 Then there exists $\left(\hat{\varphi}_k^{\mathfrak{q}_j}\right)$ an orthonormal basis of eigenfunctions of $A_{\mathfrak{q}_j}$, $j=1,2$, such that we have
\begin{equation}\label{co1}
\mu_k ^{\mathfrak{q}_1}=\mu_k^{\mathfrak{q}_2},\quad \hat{\varphi}_k^{\mathfrak{q}_1}{_{|\Omega}}=\hat{\varphi}_k^{\mathfrak{q}_2}{_{|\Omega}},\quad k\ge 1.
\end{equation}
\end{proposition}

\begin{proof} 
Noting that
\[
\int_0^{+\infty}e^{-\mu t}e^{-\mu_k^{\mathfrak{q}_j}t}dt=\left(\mu+\mu_k^{\mathfrak{q}_j}\right)^{-1},\quad k\ge 1,\; j=1,2,\; \Re\mu >0,
\]
we obtain that $\mu \in \{z\in \mathbb{C};\; \Re z>0\} \mapsto R_{\mathfrak{q}_j}(-\mu)f{_{|\Omega_1}}$ is the Laplace transform of the mapping
\begin{equation}\label{r2}
T_{\mathfrak{q}_j}(t)f=\sum_{k\ge 1}e^{-\mu_k^{\mathfrak{q}_j}t}\left(f|\varphi_k^{\mathfrak{q}_j}\right)\varphi_k^{\mathfrak{q}_j}{_{|\Omega_1}},\quad j=1,2,\;  f\in C_0^\infty(\Omega_0),\; t>0.
\end{equation}
Therefore we have
\begin{equation}\label{r1.1}
T_{\mathfrak{q}_1}(t)f= T_{\mathfrak{q}_2}(t)f,\quad t>0,\; f\in C_0^\infty(\Omega_0).
\end{equation}
 
Observe that the mapping
\[
f\in L^2(\Omega)\mapsto S_{\mathfrak{q}_j}(t)f=\sum_{k\ge 1}e^{-\mu_k^{\mathfrak{q}_j}t}\left(f|\varphi_k^{\mathfrak{q}_j}\right)\varphi_k^{\mathfrak{q}_j},\quad  t\ge 0,
\]
is nothing but the semi-group generated by $-A_{\mathfrak{q}_j}$, $j=1,2$.

Let $\left(\tilde{\mu}_k^{\mathfrak{q}_j}\right)$ be the sequence of distinct eigenvalues of $A_{\mathfrak{q}_j}$ and $\left(\tilde{\varphi}_{k,1}^{\mathfrak{q}_j},\ldots \varphi_{k,m_k^j}^{\mathfrak{q}_j}\right)$ be an orthonormal basis of the eigenspace associated to $\tilde{\mu}_k^{\mathfrak{q}_j}$, $k\ge 1$ and $j=1,2$.

In light of these new definitions we rewrite \eqref{r2} in the form
\[
T_{\mathfrak{q}_j}(t)f=\sum_{k\ge 1}e^{-\tilde{\mu}_k^{\mathfrak{q}_j}t}\sum_{\ell=1}^{m_k^j}\left(f|\tilde{\varphi}_{k,\ell}^{\mathfrak{q}_j}\right)\tilde{\varphi}_{k,\ell}^{\mathfrak{q}_j}{_{|\Omega_1}},\quad j=1,2,\;  f\in C_0^\infty(\Omega_0),\; t>0.
\]
The uniqueness of Dirichlet series yields $\tilde{\mu}_k^{\mathfrak{q}_1}=\tilde{\mu}_k^{\mathfrak{q}_2}$, $k\ge 1$, and
\[
\sum_{\ell=1}^{m_k^1}\left(f|\tilde{\varphi}_{k,\ell}^{\mathfrak{q}_1}\right)\tilde{\varphi}_{k,\ell}^{\mathfrak{q}_1}{_{|\Omega_1}}=\sum_{\ell=1}^{m_k^2}\left(f|\tilde{\varphi}_{k,\ell}^{\mathfrak{q}_2}\right)\tilde{\varphi}_{k,\ell}^{\mathfrak{q}_2}{_{|\Omega_1}},\quad k\ge 1,\;  f\in C_0^\infty(\Omega_0).
\]
We derive in a straightforward  manner from these identities 
\[
\sum_{\ell=1}^{m_k^1}\tilde{\varphi}_{k,\ell}^{\mathfrak{q}_1}{_{|\Omega_0}}\otimes\tilde{\varphi}_{k,\ell}^{\mathfrak{q}_1}{_{|\Omega_1}}=\sum_{\ell=1}^{m_k^2}\tilde{\varphi}_{k,\ell}^{\mathfrak{q}_2}{_{|\Omega_0}}\otimes \tilde{\varphi}_{k,\ell}^{\mathfrak{q}_2}{_{|\Omega_1}},\quad k\ge 1.
\]

Upon substituting $\left(\tilde{\varphi}_{k,\ell}^{\mathfrak{q}_2}\right)_{1\le \ell \le m_k^2}$ by another orthonormal basis of eigenfunctions we infer from \cite[Lemma 2.3]{CK} that  $m_k^1=m_k^2=m_k$, $k\ge 1$, and
\[
\tilde{\varphi}_{k,\ell}^{\mathfrak{q}_1}{_{|\Omega}}=\tilde{\varphi}_{k,\ell}^{\mathfrak{q}_2}{_{|\Omega}},\quad 1\le \ell \le m_k,\; k\ge 1.
\]
In other words we proved that \eqref{co1} holds.

It is worth noticing that conditions (i) and (iii) of \cite[Lemma 2.3]{CK} hold in our case as a consequence of the uniqueness of continuation property of Theorem \ref{theorem3.1}.
\end{proof}

\noindent
\textbf{Acknowledgement.} I would like to thank the referee for his valuable
comments.

\vskip .5cm
\end{document}